\documentclass [12pt,reqno]{amsart}
\usepackage{amsmath}
\usepackage{amsthm}
\usepackage{dsfont}
\newtheorem{thm}{Theorem} 
\usepackage{amsfonts, amssymb}
\usepackage[utf8]{inputenc} 
\usepackage[english]{babel}
\usepackage{mathtools}
\usepackage{ mathrsfs }
\usepackage{enumerate}
\usepackage{ amssymb }
\usepackage{cite}

\usepackage{color}

\newtheorem{theorem}{Theorem}

\newtheorem{Question}[thm]{Question}
\newtheorem*{theorem*}{Theorem}

\newtheorem{lemma}[thm]{Lemma}
\newtheorem{proposition}[thm]{Proposition}

\newtheorem{corollary}[thm]{Corollary}
\theoremstyle{definition}

\theoremstyle{definition}

\theoremstyle{definition}
\newtheorem{definition}[thm]{Definition}

\numberwithin{equation}{section}

\newcommand{\id}{\operatorname{id}}
\newcommand{\IM}{\operatorname{Im}}

\def\R{{\mathbb R}}
\def\E{{\mathbb E\,}}

\def\Z{{\mathbb Z}}

\def\N{{\mathbb N}}

\def\wt{\widetilde}

\def\PP{{\mathcal P}}
\def\PPP{{\mathscr{P}}}

\def\<{\langle}
\def\>{\rangle}
\def \ee{{\epsilon}}

\def \ll {{\lambda}}
\def \LL {{\Lambda}}
\def \H {{\mathds{H}}}

\begin{document}

\title{Finite flat spaces}


\begin{abstract}

We say that a finite metric space $X$ can be embedded almost isometrically into a class of metric spaces $C$, if for every $\epsilon > 0$ there exists an embedding of $X$ into one of the elements of $C$ with the bi-Lipschitz distortion less than $1 + \epsilon$.
We show that almost isometric embeddability conditions are equal for following classes of spaces

(a) Quotients of Euclidean spaces by isometric actions of finite groups,

(b) $L_2$-Wasserstein spaces over Euclidean spaces,

(c) Compact flat manifolds,

(d) Compact flat orbifolds,

(e) Quotients of connected compact bi-invariant Lie groups by isometric actions of compact Lie groups. (This one is the most surprising.)

We call spaces which satisfy this conditions finite flat spaces. 

Since Markov type constants depend only on finite subsets we can conclude that connected compact bi-invariant Lie groups and their quotients have Markov type $2$ with constant $1$.

\end{abstract}

\author{Vladimir Zolotov}
\address[Vladimir Zolotov]{Steklov Institute of Mathematics, Russian Academy of Sciences, 27 Fontanka, 191023 St.Petersburg, Russia, University of Cologne, Albertus-Magnus-Platz, 50923 Köln, Germany and Mathematics and Mechanics Faculty, St. Petersburg State University, Universitetsky pr., 28, Stary Peterhof, 198504, Russia.}
\email[Vladimir Zolotov]{paranuel@mail.ru}

\maketitle

\section{Introduction}
\subsection{Motivation and the main result.}
Let $X,Y$ be a pair of metric spaces. For $f:X \rightarrow Y$ the bi-Lipschitz constant of $f$ is the infimum of $c \ge 1$ s.t.,
$$\frac{1}{c}d_Y(f(x_1),f(x_2)) \le d_X(x_1,x_2) \le c d_Y(f(x_1),f(x_2)), \text{for every $x_1,x_2 \in X$}.$$

We say that a finite metric space $X$ can be embedded almost isometrically into a class of metric spaces $C$ if for every $\epsilon > 0$ the exists an embedding of $X$ into one of the elements of $C$ with the bi-Lipshitz distortion less than $1 + \epsilon$.

The study of conditions for isometric embeddability of finite metric spaces into $L_p$ spaces has a long history.  For an overview of results see \cite{DL}. In the recent years the study of isometric embeddability conditions for Alexandrov spaces of non-negative curvarute started to develop, see \cite{AKP}, \cite{ANN}, \cite{LPZ}. Despite a certain progress in understanding necessary conditions, we lack embeddability results. Which motivates the study of isometric embeddability conditions for subclasses of Alexandrov spaces with even more restricted geometry.

One can consider following subclasses: compact flat manifolds, 
quotients of Euclidean spaces by isometric actions of finite groups,
 quotients of  connected compact bi-invariant Lie groups by isometric
 actions of compact Lie groups, $2$-Wasserstein spaces over Euclidean spaces. In the present work we prove that all those classes contain the same finite subspaces. More precisely we have the following theorem.



\begin{theorem} \label{MainThm}
Let $X$ be a finite metric space. Suppose that $X$ can be almost isometrically embedded into one of the following classes of spaces. 
\begin{enumerate}
\item{$L_2$-Wasserstein spaces over Euclidean spaces, }\label{FFDWass}
\item{Quotients of Euclidean spaces by isometric actions of finite groups,}\label{FFDQFG}
\item{Compact flat orbifolds,}\label{FFDFO}
\item{Compact flat manifolds,}\label{FFDFM}
\item{Quotients of  connected compact bi-invariant Lie groups by isometric actions of compact Lie groups. \label{FFDLie}
(Here the group acting by isometries doesn't have to be a subgroup of the original group.
 We also do not assume that the group acting by isometries is connected.
In particular finite groups are fine.)}
\item{$L_2$-Wasserstein spaces over connected compact bi-invariant Lie groups.}\label{FFDLieWass}
\end{enumerate}
Then $X$ can be almost isometrically embedded into all of those classes.
\end{theorem}

We call spaces which could be almost isometrically embedded into all classes from Theorem \ref{MainThm}   \textit{finite flat spaces}.

The main claim of Theorem \ref{MainThm} is $(\ref{FFDLie}) \Rightarrow (\ref{FFDQFG})$.  In particular we have that finite subspace of 
an $n$-dimensional sphere with its standard intrinsic metrics could be almost isometrically embedded into quotients of Euclidean spaces by isometric actions of finite groups.  

Previously it was known that a bi-quotient of a compact connected Lie group with a bi-invariant metric could be presented as a quotient of Hilbert space by an isometric action of a certain group, see \cite{Devil} Problem ``Quotient of Hilbert space'', 
\cite{LPZ} proof of Proposition 1.5 and \cite{TT} Section 4.

\subsection{Question about  a synthetic definition.}
Let $T$ be a non-oriented tree with $n$ vertexes, we denote by $V(T)$ the set of its vertices and by $E(T)$ the set of its edges. 
For a metric space $X$ we say that $T$ comparison holds in $X$, if for every map $f:V(T) \rightarrow X$ there exists a map into Hilbert space $\wt f:V(T) \rightarrow \H$ s.t.,
\begin{enumerate}
\item{$d_X(f(v_1),f(v_2)) \le \vert \wt f(v_1) - \wt f(v_2)\vert $, for every $v_1, v_2 \in V(T)$.}
\item{$d_X(f(v_1),f(v_2)) = \vert \wt f(v_1) - \wt f(v_2)\vert $, for every $\{v_1,v_2\} \in E(T)$.}
\end{enumerate}

For the theory related to $T$-comparison see \cite{LPZ}. The following question is a stronger version of \cite[Question 10.2]{LPZ}. 

\begin{Question}
Let $X$ be a finite metric space such that $T$-comparison holds in $X$ for every finite tree $T$. Is is true that $X$ has to be a  finite flat space?
\end{Question}

\subsection{Application to the theory of Markov type.}

Recall that a Markov chain $\{Z_t\}^\infty_{t=0}$  with transition probabilities $a_{ij} := Pr[Z_{t+1} = j\vert Z_t = i]$ on the state space
$\{1,\dots,n\}$ is stationary if $\pi_i = Pr[Z_t = i]$ does not depend on $t$ and it is reversible if $\pi_i a_{ij} = \pi_j a_{ji}$
for every $i, j \in \{1, . . . , n\}$.

\begin{definition}Given a metric space $(X, d)$, we say that $X$ has Markov
type $2$ if there exists a constant $K > 0$ such that for every stationary reversible Markov chain
 $\{Z_t\}^\infty_{t=0}$ on $\{1,\dots,n\}$, every mapping $f : \{1, . . . , n\} \rightarrow X$ and every time $t \in \N$,
$$\E d^2(f(Z_t), f(Z_0)) \le K^2 t \E d^2(f(Z_1), f(Z_0)).$$
The least such $K$ is called the Markov type $2$ constant of $X$, and is denoted $M_2(X)$.
\end{definition}

The notion of Markov type was introduced by K. Ball in his study of the Lipschitz extension problem \cite{Ball}. 
Major results in this direction were obtained later by
Naor, Peres, Schramm and Sheffield \cite{NPSS}. The notion of Markov type has
also found applications in the theory of bi-Lipschitz embeddings \cite{BLMN, LMN}.
For more applications of the notion of Markov type and its place in a bigger picture 
see a survey \cite{RibeIntro}.  

S.-I. Ohta and M. Pichot discovered that the notion of Markov Type $2$ is related to the non-negative curvature in the sense of Alexandrov.
In \cite{OP} they showed that if a geodesic metric space has Markov type $2$ with constant $1$, then it is an Alexandrov space of non-negative curvature.  In \cite{Ohta} it was shown that every Alexandrov space has Markov type $2$ with constant $1 + \sqrt{2}$. The constant was later improved to $\sqrt{1+\sqrt{2} + \sqrt{4\sqrt{2}-1} } = 2.08\dots$ by A.Andoni, A.Naor and O. Neiman, see \cite{ANN}.

It is even possible that all Alexandrov spaces have Markov type $2$ with constant $1$. In \cite{ZMT} it was shown that some Alexandrov spaces such as compact flat manifolds, quotients of Euclidean spaces by isometric actions of finite groups and $2$-Wasserstein spaces over Euclidean spaces do have Markov type $2$ with constant $1$. The following corollary extends this list.

\begin{corollary}
Let $M$ be a quotient of a  connected compact bi-invariant Lie group by an isometric action of a compact Lie group. Then $M$ has Markov type $2$ with 
constant $1$.  In particular standard spheres with their intrinsic metrics have Markov type $2$ with constant $1$.
\end{corollary}
\begin{proof}
For a Markov chain $\{Z_t\}^\infty_{t=0}$ and  a  map $f : \{1, . . . , n\} \rightarrow M$ we have to show that 
$$\E d^2(f(Z_t), f(Z_0)) \le t \E d^2(f(Z_1), f(Z_0)),\textit{ for every $t \in \N$.}$$
Fix $\ee > 0$. By Theorem \ref{MainThm} there exists a compact flat manifold $M_\ee$ and a map $f_\ee:\{1, . . . , n\} \rightarrow M$ such that for every $i,j \in \{1, . . . , n\}$ we have
\begin{equation}\label{FMeq}
\frac{1}{(1+\ee)}d_{M_\ee}(f_\ee(i),f_\ee(j)) \le d_M(f(i),f(j)) \le (1 + \ee) d_{M_\ee}(f_\ee(i),f_\ee(j)).
\end{equation}
Since compact flat manifolds have Markov type $2$ with constant $1$ we have 
$$\E d^2(f_\ee(Z_t), f_\ee(Z_0)) \le t \E d^2(f_\ee(Z_1), f_\ee(Z_0)),\textit{ for every $t \in \N$.}$$
Combining the last inequality with (\ref{FMeq}) we obtain 
$$\E d^2(f(Z_t), f(Z_0)) \le (1+\ee)^2t \E d^2(f(Z_1), f(Z_0)),\textit{ for every $t \in \N$.}$$ 
Since $\ee$ is arbitrary we conclude that
$$\E d^2(f(Z_t), f(Z_0)) \le t \E d^2(f(Z_1), f(Z_0)),\textit{ for every $t \in \N$.}$$
\end{proof}

\section{Preliminaries and  notation}

Let $X$ be a metric space, $n \in \N$ and $\ll > 0$. We denote by $\ll X$ a metric space with a scaled metric 
$$d_{\ll X}(x,y) = \ll d_{X}(x,y),$$
by $X \times X$ a metric space on a Cartesian product given by 
$$d_{X \times X}((x_1,x_2), (y_1,y_2))^2 = d_{X}(x_1,y_1)^2 + d_{X}(x_2,y_2)^2.$$
And by $X^n$ the corresponding power of $X$,
$$X^n = X \times \dots \times X,\text{ ($n$ times)}.$$  
We denote by $S_n$ we denote a symmetric group. Group $S_n$ acts by permutations on a metric space $X^n$ we denote the 
corresponding metric quotient by $X^n\overset{perm}{/}S_n$.

For a map $f:X \rightarrow Y$ between metric spaces $X$ and $Y$ we denote by $\vert f \vert_{Lip}$ the Lipschitz constant of $f$ i.e.,
$$\vert f \vert_{Lip}  = \sup_{x_1,x_2 \in X}\frac{d_Y(f(x_1),f(x_2))}{d_X(x_1,x_2)} \in [0,\infty].$$

Now we are going to recall the definition of $2$-Wasserstein spaces. For a metric space $X$ we denote by $\PPP(X)$ the set of 
Borel probability measures with finite $2$-nd moment which means that
 $$ \exists  o \in X: \int_Xd^2(x,o)d\mu(x) < \infty.$$
Let $\mu,\nu \in \PPP(X)$. We say that a measure $q$ on $X \times X$ is a coupling of $\mu$ and $\nu$ iff its marginals are $\mu$ and $\nu$, that is 
$$q(A \times X) = \mu(A), q(X \times A) = \nu(A),$$
for all Borel measurable subsets $A \subset X$.
The $2$-Wasserstein distance $d_{W^2}$ between $\mu$ and $\nu$ is defined by 
$$d_{W_2}(\mu,\nu) = \inf\Big\{\Big(\int_{X \times X}d^2(x,y)dq(x,y)\Big)^\frac{1}{2}: \text{$q$ is a coupling of $\mu$ and $\nu$ }\Big\}.$$
The $2$-Wasserstein space $\PPP(X)$ is the set of Borel probability measures with finite $2$-nd moment on $X$ equipped with $2$-Wasserstein distance.  

\section{Proof of Theorem \ref{MainThm}}
The scheme of proof is cyclic $(\ref{FFDWass}) \Rightarrow (\ref{FFDQFG}) \Rightarrow (\ref{FFDFO}) \Rightarrow (\ref{FFDFM}) \Rightarrow (\ref{FFDLie}) \Rightarrow (\ref{FFDLieWass})  \Rightarrow (\ref{FFDWass})$ and the only really new arrow is  $(\ref{FFDLieWass})  \Rightarrow (\ref{FFDWass})$.

The arrow $(\ref{FFDWass}) \Rightarrow (\ref{FFDQFG})$ is a direct implication of the following  observation which is due to Sergey  V. Ivanov and appears in \cite{ZMT}, see Lemma 6.1 and the proof of Proposition 6.2.
\begin{lemma}
For a metric space $X$ there exist a sequence of isometric embeddings $${\Phi_n:X^{2^n}\overset{perm}{/}S_{2^n} \rightarrow \PPP(X)}$$ such that images $I_n := \Phi_n\big(X^{2^n}\overset{perm}{/}S_{2^n}\big)$   satisfy, 
\begin{enumerate}
\item{
$I_n  \subset I_{n+1},\text{ for every $n \in \N$},$
}
\item{
$\cup_{n = 1}^{\infty}I_n$ is dense in $\PPP(X).$
}
\end{enumerate}
\end{lemma}

Now we are going to provide $(\ref{FFDQFG})  \Rightarrow (\ref{FFDFO})$. 
Let  $X$  be a finite subspace of a quotient space $\R^n/G$, 
where $G$ is a finite group acting by isometries.
There exists an Euclidean space $\R^m$ and an action $\rho_1$ of $G$ on $\R^m$  
by permutation of coordinates such that 
$X$ can be embedded isometrically into $\R^m/\rho_1$, 
see for example \cite[Corollary 1]{ZS}. 
For $M > 0$ we denote by $\rho_2^{(M)}$ an action of $\Z^m$ on $\R^m$ 
by shifts scaled by $M$, i.e.
$$(\rho_2^{(M)}(a_1,\dots,a_m))(x_1,\dots,x_m) = (x_1 + Ma_1,\dots, x_m + Ma_m).$$
Note that the product action $\rho_1 \times \rho_2^{(M)}$ 
is discrete and the fundamental domain is bounded. 
Thus $\R^m/(\rho_1 \times \rho_2^{(M)})$ is a compact flat orbifold.
 If $M$ is big enough then $\R^m/(\rho_1 \times \rho_2^{(M)})$ contains an isometric copy of $X$.

 An implication $ (\ref{FFDFO}) \Rightarrow (\ref{FFDFM})$  follows from the next proposition, see \cite[Proposition 3.3]{LytchaksFriends}.
\begin{proposition} 
 Any flat orbifold is the Gromov-Hausdorff limit of a sequence of closed flat manifolds.
\end{proposition}

The  next arrow is $(\ref{FFDFM}) \Rightarrow (\ref{FFDLie})$. By Bieberbach’s Theorem \cite{Bie1,Bie2} a flat manifold can be presented as a quotient of a flat torus by an isometric action of a finite group.  Thus, we have $(\ref{FFDFM}) \Rightarrow (\ref{FFDLie})$.

To provide $ (\ref{FFDLie}) \Rightarrow (\ref{FFDLieWass})$ we simply apply the following proposition, see  \cite[Theorem 3.2]{plift}.

\begin{proposition}\label{LiftLem}
Let $M$ be a compact Riemannian manifold and $\rho:G \rightarrow Iso(M)$ be an action by isometries of a compact Lie group $G$ on $M$.
Let $N$ denotes the corresponding quotient space. 
There exists a lifting map 
$$\LL: \PP(N) \rightarrow \PP(M),$$ 
such that for every $\mu, \nu \in \PP(N)$  we have 
$$d_{W_2}(\LL(\mu),\LL(\nu)) = d_{W_2}(\mu, \nu),$$
$$\rho_{\sharp}(\LL(\mu)) = \mu,$$
where $\rho_{\sharp}:M \rightarrow N$ is the projection.
\end{proposition}

Finally we are going to provide the proof of $(\ref{FFDLieWass})  \Rightarrow (\ref{FFDWass})$.  The following observation and its relation to our study were shown to the author by Alexander Lytchak.  

\begin{lemma}\label{LytchaksTrick}
Let $G$ be a  connected  bi-invariant Lie group. 
Consider an action $${\rho:G \rightarrow Iso(\sqrt{2} G \times \sqrt{2} G)},$$ given by 
$$\rho(g)(g_1,g_2) = (gg_1,gg_2).$$
Then the corresponding quotient space $(\sqrt{2} G \times \sqrt{2} G)/G$ is isometric to $G$. 
\end{lemma}
Proof of Lemma \ref{LytchaksTrick} is relatively straightforward and we omit it.

We denote by $E^n$ the Euclidean space of the dimension $n$. 
We are going to construct a sequence of maps $\wt \LL_m : \PPP(G) \rightarrow \PPP(E^{n(m)})$ indexed by positive integers $m = 1,2,3\dots$, such that bi-Lipschitz distortions of $\wt \LL_m$ tend to zero. 

\textit{Step 1: Construction of maps $\wt \LL_m$.} For $m \in \N$ we denote $M = M(m) = 2^{\frac{m}{2}}$.
Lemma \ref{LytchaksTrick} provides us a tower of groups
$$\dots \overset{p_{m+1}}{\rightarrow} (MG)^{M^2} \overset{p_m}{\rightarrow} \dots \overset{p_3}\rightarrow 2G \times 2G \times 2G \times 2G \overset{p_2}{\rightarrow}  \sqrt{2}G \times \sqrt 2 G \overset{p_1}{\rightarrow} G.$$
By Proposition \ref{LiftLem} we can construct a tower of lifting maps, 
$$ \dots \overset{\LL_{m+1}}{\leftarrow} \PPP((MG)^{M^2}) \overset{\LL_m}{\leftarrow} \dots \overset{\LL_3}{\leftarrow} \PPP(2G \times 2G \times 2G \times 2G) \overset{\LL_2}{\leftarrow}  \PPP(\sqrt{2}G \times \sqrt 2 G) \overset{\LL_1}{\leftarrow} \PPP(G).$$


For a metric space $X$, some Euclidean space $E$, 
a map $A:X\rightarrow E$, and $C > 0$ 
we denote by $A^{(C)}: CX \rightarrow E$ a map given by $A^{(C)}(x) = CA(x)$.
By the  Nash embedding theorem (see \cite{Nash}) there exist
$k \in \N$ and a bijective  Riemanian isometric $C^1$-map $f:G \rightarrow E^k$.
We define a map $F_m:(MG)^{M^2}\rightarrow E^{kM^2}$  by  $$F_m(g_1,\dots,g_{M^2}) =  (f^{(M)}(g_1),\dots,f^{(M)}(g_{M^2})).$$

The required map $\wt \LL_m:\PPP(G) \rightarrow \PPP(E^{kM^2})$ is defined by $$\wt \LL_m = {(F_m)}_\sharp \circ \LL_m \circ \dots \circ \LL_1.$$

\textit{Step 2: Proving that bi-Lipschitz distortions of $\wt \LL_m$ tend to $1$.} 
Note that    $ \LL_1, \dots , \LL_m$ and  ${(F_m)}_\sharp$ are $1$-Lipschitz. Thus, $\wt \LL_m$ is also $1$-Lipschitz.

 We introduce a map $\wt p_m : \IM(F_m) \rightarrow G$ given  by
 $\wt p_m = p_1 \circ \dots \circ p_m  \circ F_m^{-1}.$  
 Note that $ (\wt p_m)_\sharp \circ \wt \LL_m =  \id $.
Thus to show that bi-Lipshitz distortions of $\wt \LL_m$ tend to $1$ 
it suffices to show that 
 $\lim_{m \rightarrow \infty}\vert \wt p_m\vert_{Lip} \le 1$.

We denote by $D$ the diameter of $G$  and by 
$\vert \cdot \vert_{E^n}$ the standard norm in $E^n$. 
For a pair of points  $x,y \in \IM(F_m)$ s.t, $\vert x - y\vert_{E^{kM^2}} > d$ 
we clearly have 
$$ \vert x - y \vert_{E^{kM^2}} >  D  \ge  d_G(\wt p_m(x), \wt p_m(y)).$$

 Note that since $f$ is a Riemannian isometric $C^1$-map and $G$ is compact there exists $L > 0$, s.t
$\vert f^{-1}\vert_{Lip} < L$.  From the construction of $F_m$ we obtain that $\vert {F_m}^{-1}\vert_{Lip} < L$ for every $m \in \N$.

Thus, for every pair of points 
$x,y \in  \IM(F_m)$ s.t, $\vert x - y\vert_{E^{kM^2}} \le  D$ we have 
$$d_G({F_m}^{-1}(x),{F_m}^{-1}(y)) < L D.$$

Thus, 
$$\vert x - y\vert_{E^{kM^2}} \ge d_{(MG)^{M^2}}({F_m}^{-1}(x),{F_m}^{-1}(y)) 
                       \inf_{\wt x, \wt y \in (MG)^{M^2}, d_{(MG)^{M^2}}(\wt x,\wt y) < LD}
					   \frac{\vert F_m(\wt x) - F_m(\wt y)\vert_{E^{kM^2}}}{d_{(MG)^{M^2}}(\wt x, \wt y)} \ge $$
$$\ge   d_G({\wt p_m}(x),{\wt p_m}(y)) 
                      \inf_{\wt x, \wt y \in (MG)^{M^2}, d_{(MG)^{M^2}}(\wt x,\wt y) < LD}
					   \frac{\vert F_m(\wt x) - F_m(\wt y)\vert_{E^{kM^2}}}{d_{(MG)^{M^2}}(\wt x, \wt y)} \ge $$
$$\ge   d_G({\wt p_m}(x),{\wt p_m}(y)) 
                       \inf_{\wt x, \wt y \in (MG), d_{MG}(\wt x,\wt y) < LD}
					   \frac{\vert f^{(M)}(\wt x) - f^{(M)}(\wt y)\vert_{E^{k}}}{d_{MG}(\wt x, \wt y)} = $$
$$=   d_G({\wt p_m}(x),{\wt p_m}(y)) 
                       \inf_{\wt x, \wt y \in G, d_G(\wt x,\wt y) < \frac{LD}{M}}
					   \frac{\vert f(\wt x)  - f(\wt y)\vert_{E^{k}}}{d_{G}(\wt x, \wt y)}.$$

Note that since $f$ is a Riemannian isometric $C^1$-map and $G$ is compact 
$$\inf_{\wt x, \wt y \in G, d_G(\wt x,\wt y) < \frac{LD}{M}}\frac{\vert f(\wt x) - f(\wt y)\vert_{E^{k}}}{d_{G}(\wt x, \wt y)} \underset{m \rightarrow \infty}{\rightarrow} 1.$$

\subsection*{Acknowledgements}
I thank Sergey V. Ivanov and Alexander Lytchak for advising me during this work. 
I'm grateful to Nina Lebedeva for fruitful discussions.
I thank the anonymous referee for numerous corrections. 
The paper is supported by the Russian Science Foundation under grant 16-11-10039.

\bibliography{circle}

\bibliographystyle{plain}

\end{document}